\documentclass[a4paper,12pt]{article}

\usepackage{amsmath,amsthm,amssymb,sansmath,color,geometry,multicol,graphicx, float,color,authblk, hyperref,stmaryrd,enumerate}
\usepackage[english]{babel}
\usepackage[utf8]{inputenc}
\newtheorem{thm}{Theorem}
\newtheorem{prop}[thm]{Proposition}

\newtheorem{lem}[thm]{Lemma}
\newtheorem{cor}[thm]{Corollary}

\newtheorem*{defi*}{Definition}
\newtheorem{assu}{Assumption}

\newcommand{\po}{\left(}
\newcommand{\pf}{\right)}
\newcommand{\co}{\left[}
\newcommand{\cf}{\right]}
\newcommand{\cco}{\llbracket}
\newcommand{\ccf}{\rrbracket}
\newcommand{\R}{\mathbb R}

\newcommand{\W}{\mathcal W}
\newcommand{\dd}{\text{d}}

\newcommand{\na}{\nabla}

\newcommand\ent[2]{\mathcal H \left(\left.#1 \right|#2\right)}

\title{Long-time behaviour and propagation of chaos for mean field kinetic particles}
 \author{Pierre Monmarch\'e\footnote{Universit\'e de Neuch\^atel, rue Emile Argand 11, 2000 Neuch\^atel, Switzerland}}
\date{August 18, 2016}

\begin{document}

\maketitle

\begin{abstract}
The trend to equilibrium in large time is studied for a large particle system associated to a Vlasov-Fokker-Planck equation  in the presence of a convex external potential, without smallness restriction on the interaction. From this are derived uniform in time propagation of chaos estimates, which themselves yield in turn an exponentially fast convergence for the semi-linear equation itself. The approach is quantitative. 
\end{abstract}

\section{Settings and main results}

The Vlasov-Fokker-Planck equation reads
\begin{eqnarray}\label{EqVlasovFP}
\lefteqn{\partial_t m_t +y \cdot \na_x m_t\ =}\notag\\  &  & \na_y \cdot \po \frac{\sigma^2}{2} \na_y m_t + \po \int \na_x U\po x, u \pf   m_t(u,v) \dd u \dd v + \gamma y\pf m_t\pf
\end{eqnarray}
where $m_t(x,y)$ is a density at time $t$ of particles at point $x\in \R^d$ with velocity $y\in\R^d$, $\sigma,\gamma>0$, $\na$ and $\na\cdot$ stand for the gradient and divergence operators and the potential $U$ is a function from $\R^{2d}$  to $\R$. 
This equation is naturally linked to the stochastic system of $N$ kinetic particles $(X_i,Y_i)_{i\in\cco 1,N\ccf}$ that solves 
\begin{eqnarray}\label{EqSystemparticul}
\forall i\in \cco 1,N\ccf & &\left\{\begin{array}{rcl}
\dd X_i & = & Y_i \dd t\\
\dd Y_i & = & - \gamma Y_i \dd t - \po \frac{1}{N}\underset{j=1}{\overset{N}\sum}  \na_x U \po X_i, X_j\pf \pf\dd t + \sigma \dd B_i 
\end{array} \right. 
\end{eqnarray} 
with the initial conditions $(X_i(0),Y_i(0))$ being i.i.d. random variables of law $m_0$, independent from the standard Brownian motion $B$. Here, \emph{naturally linked} means, depending on one's point of view, that for large $N$ the semi-linear PDE is an approximation of the particle system, or that the particle system is an approximation of the semi-linear PDE:  for large $N$,  the random empirical law $M_t^{N}=\frac1N \sum_{i\in \cco 1,N\ccf} \delta_{X_i,Y_i}$ (where $\delta_{x,y}$ is the Dirac mass at point $(x,y)$) should behave like the deterministic solution $m_t$ of \eqref{EqVlasovFP}. 

Equation \ref{EqSystemparticul} describes a mean field dynamics in the sense each  particle is equally influenced by all the others at once.  As $N$ goes to infinity, two given particles should be less and less correlated, and a given particle $(X_i,Y_i)$ should behave like $(\overline{X}_i,\overline{Y}_i)$ which solves
\begin{eqnarray}\label{EqParticulNonLinea}
  & &\left\{\begin{array}{rcl}
\dd \overline{X}_i  &= &\overline{Y}_i \dd t\\
\dd \overline{Y}_i &= & - \gamma \overline{Y}_i   \dd t - \po \int  \na_x U\po \overline{X}_i, u\pf  m_t(u,v)\dd u\dd v\pf\dd t + \sigma \dd B_i
\end{array} \right. 
\end{eqnarray}
with $(\overline{X}_i(0),\overline{Y}_i(0))= (X_i(0),Y_i(0))$. This is the so-called propagation of chaos phenomenon, and it is equivalent for interchangeable particles to the convergence of the empirical measure of the particle system to $m_t$ (see \cite{Sznitman}). 
 Note the $(\overline{X}_i,\overline{Y}_i)$'s are i.i.d. random variables with law $m_t$.

Aside from the propagation of chaos ($N\rightarrow\infty$) another natural question is the long time  behaviour ($t\rightarrow\infty$) of both $m_t$ and  $(X_i,Y_i)_{i\in\cco 1,N\ccf}$. There is a large literature for the corresponding space homogeneous model (the Mc Kean-Vlasov equation, which can also be seen as an overdamped version of \eqref{EqVlasovFP} when the particles have no mass), or when there is no interaction (i.e. when $U(x,c)$ only depends on $x$, which is the kinetic Fokker-Planck or Langevin equation); see the introductions of \cite{BolleyGuillinMalrieu,BolleyGentilGuillin} for references.

However obtaining explicit speed of convergence toward equilibrium for the Vlasov-Fokker-Planck equation (at least when there is a unique equilibrium, which is false in general, see \cite{Tugaut2015}), coping with both the kinetic dynamics and the interaction, is more  difficult.   There are several results for small interactions  \cite{HerauThomann,Herau2007,Villani2009} or when the forces are close to be linear \cite{BolleyGuillinMalrieu}. See also  \cite{Esposito} for a case with possibly several equilibria.

In the recent work of H\'erau and Thomann \cite{HerauThomann}, since the interaction is small, exponential convergence to equilibrium is established from the results on the linear Fokker-Planck equation by considering the semi-linearity as a perturbation. Our strategy is also to use the linear case, but by replacing the non-linearity by interaction in large dimension (following the ideas of Malrieu \cite{Malrieu2001}). As a consequence we replace the smallness condition of the interaction by convexity assumptions:

\begin{assu}\label{Hypo1} The potential $U(x,u) = V(x) + V(u) + W(x-u)$, where $V$ and $W$ are smooth with all their derivatives of order larger than 2 bounded. The exterior potential $V$ is strictly convex in the sense there exists $c_1>0$ such that for all $x,u\in\R^d$, $u\cdot \na^2 V(x) u \geq c_1|u|^2$ where $\na^2 V$ is the Hessian matrix of $V$. The interaction potential $W$ is even and there exists $c_2<\frac12c_1$ such that  for all $x,u\in\R^d$, $u\cdot \na^2 W(x) u \geq -c_2|u|^2$.

The initial law $m_0$ admits a smooth density in $L\log L$ with respect to the Lebesgue measure and a finite second moment. 
\end{assu}
Note that the term $V(u)$ in the definition of $U(x,u)$ does not alter the force $\na_x U(x,u)$. It is added  only for the sake of symmetry and of the definition of the global potential $U_N$ in Section~\ref{SectionSystemedeParticule}.

These conditions discard Coulomb interaction forces, but it treats mollified approximations of those, as in \cite{Herau2007}.  Under Assumption \ref{Hypo1}, as we will see, the solution $m_t$ of \eqref{EqVlasovFP} is well defined, along with the processes $Z_N=\po (X_i,Y_i)\pf_{i\in\cco 1,N\ccf}$ and  $\overline Z_N=\po (\overline X_i,\overline Y_i)\pf_{i\in\cco 1,N\ccf}$. We call $m_t^{(N)}$ the law of $Z_N$ when $m_0^{(N)} = m_0^{\otimes N}$, while the law of $\overline Z_N$ is $m_t^{\otimes N}$. The process $Z_N$ admits a unique invariant law, denoted by $m_\infty^{(N)}$. For two probability laws $\nu$ and $\mu$ we write
\[\ent{\nu}{\mu} = \left\{ \begin{array}{ll}
\int \ln \po \frac{\dd \nu}{\dd \mu}\pf \dd \nu & \text{if }\nu \ll \mu\\
+\infty & \text{else}
\end{array}\right.\]
the relative entropy of $\nu$ with respect to $\mu$. Under Assumption \ref{Hypo1}, at initial time, $\ent{m_0^{(N)}}{m_\infty^{(N)}} <\infty$ (more precisely, see Lemma \ref{LemPropMalrieu} below).

\begin{thm}\label{TheoNLine}
Under Assumption \ref{Hypo1}, there exist $C,\chi>0$ that depends only on $U,\gamma,\sigma$ such that for all $N$ and $t\geq 0$,
\begin{eqnarray}
\ent{m_t^{(N)}}{m_\infty^{(N)}} & \leq & Ce^{-\chi t} \ent{m_0^{(N)}}{m_\infty^{(N)}}.
\end{eqnarray}
\end{thm}
Note that $m_t^{(N)}$ is a solution of a (large dimensional) linear Fokker-Planck equation, for which the exponential decay of the entropy is already known (see e.g. \cite{Villani2009}). The key point in Theorem~\ref{TheoNLine} is that $C$ and $\chi$ do not depend on $N$. This enables us to prove the following:

\begin{cor}\label{CorPW2}
Under Assumption  \ref{Hypo1}, Equation \eqref{EqVlasovFP} admits a unique equilibrium $m_\infty$. Moreover, for any $\beta>d+2$ there exist $K,N_0 >0$ that depends only on $U,\gamma,\sigma,\beta$ and $m_0$ such that for all   $\varepsilon,t$ and $N \geq N_0 \po 1 \wedge \varepsilon\pf^{-\beta}$,
\begin{eqnarray*}
\mathbb P\po \mathcal W_2\po M_t^N , m_\infty \pf\geq \varepsilon \pf & \leq &  \frac{K}{\varepsilon^2}\po e^{-\chi t} +\frac{1}{ N} \pf
\end{eqnarray*}
where $\chi$ is given by Theorem \ref{TheoNLine} and $\mathcal W_2$ is the Wasserstein distance
\begin{eqnarray*}
\mathcal W_2^2 \po \nu_1,\nu_2\pf &=&  \inf \left\{\mathbb E\po|A_1-A_2|^2\pf,\ Law(A_i)= \nu_i \text{ for }i=1,2\right\}.
\end{eqnarray*}
\end{cor}
(We emphasize that throughout this work $|\cdot|$ always stands for the usual Euclidean distance, and not a renormalization of it, as in some other works on propagation of chaos.)

As shown in \cite[Proposition 2.1]{BolleyGuillinVillani}, such a result yields confidence intervals with respect to the uniform metric for a numerical approximation of $m_\infty$ by $M_t^N \ast \xi$ where $\xi$ is a smooth kernel. Since $m_\infty$ turns to be the tensor product of an explicit Gaussian distribution in velocity and of the invariant measure of an order one McKean-Vlasov equation (such as considered in \cite{BolleyGuillinVillani,Malrieu2001}) in position, there are thus two possibilities for its approximation, and we may wonder if one is better than the other.  This question already arises to sample a Gibbs law $e^{-U}$ even when there is no interaction and $U$ is explicit: this can be done either with a Fokker-Planck reversible process, or with a kinetic Langevin non-reversible one. The reversible dynamics is simpler but the kinetic Langevin sampler may converge to equilibrium faster. This has been observed numerically in \cite{Lelievre2006},  proven for some toy models in \cite{Gadat2013} and there are some philosophical reasons to believe it, but it is hard to state in general since the rates of convergence obtained for so-called hypocoercive processes are usually not sharp.

From Theorem \ref{TheoNLine}, one can also expect to recover a rate of convergence for the deterministic equation \ref{EqVlasovFP}. This requires explicit estimates for the propagation of chaos that behave nicely with respect to time, which are not so easy to establish. In the close to linear case (and with a small interaction), Bolley, Guillin and Malrieu in \cite{BolleyGuillinMalrieu} are able to use the exterior coercive force to prove uniform in time estimates directly from parallel coupling, namely by considering $Z_N$ and $\overline Z_N$ which respectively solve \eqref{EqSystemparticul} and \eqref{EqParticulNonLinea} driven by the same Brownian motion. In our case, this does not work. However, using the parallel coupling for small times and the coupling of the equilibria for large times, an intertwining between results of propagation of chaos and of large time convergence ultimately yields the two following consequences:

\begin{thm}\label{ThmChaosUniforme}
Under Assumptions \ref{Hypo1}, there exist $\alpha>0$ (depending only on $U,\gamma$ and $\sigma$)  and $K>0$ (depending only on $U,\gamma,\sigma$ and $m_0$) such that for all $N\geq 1, \ t>0$, 
\begin{eqnarray*}
\mathcal{W}_2\po m_t^{(1,N)},m_t \pf & \leq & \frac{K}{N^\alpha}
\end{eqnarray*}
where $m_t^{(1,N)}$ stands for the first $2d$-dimensional marginal of $m_t^{(N)}$, namely for the law of $(X_1,Y_1)$.
\end{thm}

(And in fact we also get a uniform in time propagation of chaos estimate for the total variation distance, see the last remark of the paper.)






\begin{thm}\label{TheoNonLineaire}
Under Assumptions \ref{Hypo1}, there exist $\chi'>0$ (depending only on $U,\gamma$ and $\sigma$) and  $K >0$ (depending only on $U,\gamma,\sigma$ and $m_0$) such that for all $t\geq 0$,
\begin{eqnarray}
\| m_t - m_\infty \|_1 & \leq &  K  e^{-\chi' t} .
\end{eqnarray}
\end{thm}


Note that uniform in time propagation of chaos cannot hold when the non-linear dynamics admits several equilibria, since $m_t$ will converge in large time to one of those while the invariant law of $m_t^{(N)}$  is unique and should converge as $N\rightarrow\infty$ to a combination of all of these equilibria. Nevertheless we may hope the same kind of arguments to work when $m_\infty^{(N)}$ is replaced by a quasi-stationary distribution (QSD) of the particle system, and the convergence of $t$ and $N$ to infinity is more intricate, namely for a given $N$ we let $t$ be large enough so that $m_t^{(N)}$ converges to its QSD but not so large so that it doesn't leave the catchment area of a particular equilibrium (to which the QSD should converge as $N\rightarrow\infty$). This is the subject of ongoing research.

\bigskip

The rest of the paper is organized as follow:  Theorem \ref{TheoNLine}, Corollary \ref{CorPW2}, Theorem \ref{ThmChaosUniforme} and Theorem \ref{TheoNonLineaire} are respectively proven in Sections \ref{SectionSystemedeParticule}, \ref{SectionIntervalle}, \ref{SectionChaos} and \ref{SectionVFP}. 

\bigskip

\textbf{Notation:} throughout all the paper, from lines to lines, we keep denoting by $K$ different constants as long as they depend only on $U,\gamma,\sigma$ and $m_0$.

\section{The particle system}\label{SectionSystemedeParticule}


Let us recall some known facts whose proofs may be found in \cite{Talay,MonmarcheRecuitHypo}. The SDE \eqref{EqSystemparticul} admits a unique strong solution defined for all times. Denoting by $P_t^N$ its associated semi-group, defined on (say bounded) functions on $\R^{2dN}$ by $P_t^N f(z) = \mathbb E\po f(Z_N(t))\ |\ Z_N(0) = z\pf$ (recall we write $Z_N = \po (X_i,Y_i)\pf_{i\in\cco 1,N\ccf}$), then $P_t^N$ is strongly Feller and fixes the set of smooth functions whose all derivatives grow at most polynomially. The process admits a unique invariant probability measure whose density with respect to the Lebesgue measure is 
\[m_\infty^{(N)}(x,y) \ \propto \  \exp\po -\frac{2\gamma}{\sigma^2}\po U_N(x) + \frac{|y|^2}{2}\pf\pf\]
where the full potential
\[U_N(x) := \frac1{2N}\sum_{i,j=1}^N U(x_i,x_j)\]
is such that the term that depends on $X$ in \eqref{EqSystemparticul}  in $d Y_i$ is exactly
\[\nabla_{x_i} U_N(x) = \frac1N\sum_{j=1}^N \nabla U(x_i,x_j).\]
\begin{lem}\label{LemHessienneUN}
Under Assumption \ref{Hypo1}, for all $N\in\mathbb N$ and all $x,u\in\R^{dN}$, 
\[\po c_1 - 2 c_2\pf |u|^2 \leq u\cdot \na^2 U_N(x) u \leq \po \|\na^2 V\|_\infty+ 2\|\na^2 W\|_\infty \pf |u|^2 . \]
\end{lem}

\begin{proof}
This is a direct consequence of 
\begin{eqnarray*}
 u \cdot   \na^2 U_N(x) u  & = & \sum_{i=1}^N u_i \cdot  \na^2 V(x_i) u_i\\
 & &  + \frac{1}{2N} \sum_{i,j=1}^N (u_i - u_j) \cdot  \na^2 W(x_i-x_j)(u_i - u_j).
\end{eqnarray*}
\end{proof}
Recall we say a measure $\mu$ satisfies a log-Sobolev inequality with constant $\eta>0$ if
\begin{eqnarray}\label{EqLogSob}
\forall f>0\ s.t.\ \int f\dd \mu = 1,\hspace{40pt} \int f\ln f\dd \mu  & \leq & \eta \int  \frac{|\na f|^2}{f} \dd \mu.
\end{eqnarray}

\begin{lem}\label{LemSobolevmN}
Under Assumption \ref{Hypo1}, the measure $m_\infty^{(N)}$ satisfies a log-Sobolev inequality with a constant $\eta$ that does not depend on $N$.
\end{lem}
\begin{proof}
For $z\in\R^{2dN}$, let $F(z) = \frac{2\gamma}{\sigma^2}\po U_N(x) + \frac{|y|^2}{2}\pf$. For all $z,u\in\R^{2dN}$ with $|u|=1$,
\begin{eqnarray*}
u\cdot \na^2 F(z) u &\geq& \frac{2\gamma}{\sigma^2}\min \po c_1-2c_2,1\pf \ := \frac{1}{2\eta} \ >\ 0.
\end{eqnarray*}
Since $m_\infty^{(N)}  \propto e^{-F}$ is the invariant measure of the semi-group with generator $-\na F\cdot \na + \Delta$, the Bakry-Emery curvature criterion (see \cite{BolleyGentil}) concludes. 
\end{proof}
Consider the generator
\begin{eqnarray}\label{EqDefiLN}
L_N   &=& -y\cdot\na_x   + \po\na U_N(x)-\gamma y\pf \cdot \na_y   + \frac{\sigma^2}{2} \Delta_{y}.
\end{eqnarray}
 Then $h_t^{(N)} = \frac{ m_t^{(N)}}{  m_\infty^{(N)}}$, the density of the law of the particle system \eqref{EqSystemparticul} with respect to its equilibrium, solves $\partial_t h_t^{(N)} = L_N h_t^{(N)}$. This is a linear kinetic Fokker-Planck equation, for which convergence to equilibrium has been proven by many ways.  All we need to check is that the explicit estimates we obtain do not depend on $N$. For instance we can use the following

\begin{thm}[from Theorem 10 of  \cite{MonmarcheGamma}]\label{Theo10Gamma}
Consider a diffusion generator $L$  on H\"ormander form
\[L = B_0 + \sum_{i=1}^d B_i^2\]
where the $B_j$'s are derivation operators. Suppose there exist $N_c\in \mathbb N$ and $\lambda,\Lambda,m,\rho,K>0$ such that for $i\in\llbracket 0,N_c+1\rrbracket$ there exist smooth derivation operators $C_i$ and $R_i$ and a scalar field $Z_i$  satisfying:
\begin{enumerate}[(i)]
\item $C_{N_c+1} = 0$,\ and \  $[B_0,C_i] = Z_{i+1} C_{i+1} + R_{i+1}$ \  for all $i\in\llbracket 0,N_c\rrbracket$, where $[A,B]=AB-BA$ stands for the Poisson bracket of two operators,
\item $[B_j,C_i] = 0$ \ for all $i\in\llbracket 0,N_c\rrbracket$, $j\in \llbracket 1,d\rrbracket$, 
\item  $\lambda \leq  Z_i  \leq \Lambda$ \ for all $i\in\llbracket 0,N_c\rrbracket$,
\item $|C_0 f|^2 \leq m \underset{j\geq 1}\sum |B_j f|^2$ and $|R_i f|^2 \leq m \underset{j<i}\sum |C_j f|^2$ for all $i\in\llbracket 0,N_c+1\rrbracket$ and smooth Lipschitz $f$.
\item $\underset{i\geq 0}\sum | C_i f|^2 \geq \rho |\na f|^2$.
\end{enumerate}
Suppose moreover that there exists a probability measure $\mu$ which is invariant for $e^{tL}$ and satisfies a log-Sobolev inequality with constant $\eta$.

\bigskip

Then   for all $t>0$ and for all $f>0$ with $\int f\dd \mu = 1$,
\begin{eqnarray}\label{EqTheorCoerc}
 \int \po e^{tL}f\pf\ln \po e^{tL}f\pf\dd \mu & \leq &  e^{-\kappa t(1-e^{-t})^{2N_c}}  \int f\ln f\dd \mu
\end{eqnarray}
with
\begin{eqnarray*}
\kappa &=& \frac{\rho}{\eta} \po \frac{100}{\lambda}\po N_c^2 + \frac{\Lambda^2}{\lambda}+m\pf \pf^{-20 N_c^2}.
\end{eqnarray*}
\end{thm}

\begin{proof}[Proof of Theorem \ref{TheoNLine}]
The generator \eqref{EqDefiLN} is on H\"ormander form 
\[B_0 + \sum_{i=1}^N \sum_{j=1}^d B_{i,j} \]
with, writing $y_i = \po y_i^{(1)},\dots,y_i^{(d)}\pf\in\mathbb R^d$,
\begin{eqnarray*}
B_0 &=& -y\cdot\na_x   + \po\na U_N(x)-\gamma y\pf \cdot \na_y   \\
B_{i,j} &=& \frac{\sigma}{\sqrt 2} \partial_{y_i^{(j)}}.
\end{eqnarray*}
 Since
 \[ [B_0,\na_y] =[L_N,\na_y] = \na_x + \gamma\na_y, \hspace{35pt}[B_0,\na_x] =[L_N,\na_x] = - \na_x^2 U_N \na_y,\]
 Theorem \ref{Theo10Gamma} applies with 
 \begin{eqnarray*}
C_0 = \na_y,& \hspace{20pt} &C_1 = \na_x,\\
R_1 = \gamma \nabla_y, & \hspace{20pt} &R_2 = - \na^2 U_N \na_y,
 \end{eqnarray*}
  \[Z_1=Z_2=N_c=\lambda=\Lambda=\rho=1,\]
   \[m= \frac2{\sigma^2 } +  \gamma^2 +\po\|\na^2 V\|_\infty+2\|\na^2 W\|_\infty\pf^2\] and $\eta$ given by Lemma \ref{LemSobolevmN}.
 \end{proof}

Remark that instead of \cite[Theorem 10]{MonmarcheGamma}, we could have referred to the work of Villani, which is anterior, to get the same result. Indeed, combining the Theorems 28 and A.15 of \cite{Villani2009} yields a similar hypocoercive convergence in the relative entropy sense. Nevertheless, this would have required a thorough examination of the somewhat tedious computations in the  proofs of these theorems, in order to  explicit the constants $C$ and $\chi$ and  check that they do not depend on $N$. Besides, since we did not conduct such an examination, it is not clear whether the computations in \cite{Villani2009}, which are known not to be sharp, can indeed  give constants $C$ and $\chi$ which are uniform. Yet, this is the crucial argument on which relies all the rest of the present work.

Remark also that, on the contrary, we could not use the method of Dolbeault, Mouhot and Schmeiser (see the seminal work \cite{DMS2009}) since, rather than in the relative entropy sense, the latter states a convergence in the $L^2$ distance sense, which behaves badly with respect to tensorization. More precisely, $\mathcal H(m_1\ | \ m_2)$ would be replaced by
\begin{eqnarray*}
\int \po \frac{m_1}{m_2} - 1\pf^2 m_2 &=& \int \po \frac{m_1}{m_2} \pf^2 m_2 - 1 
\end{eqnarray*}
and thus, for independent variables, $\mathcal H\po m_1^{\otimes N}\ |\ m_2^{\otimes N}\pf$ would be replaced by
\[ \po \int \po \frac{m_1}{m_2} \pf^2 m_2\pf^N - 1 .\]

\bigskip

In our case, the explicit bound on the rate given by Theorem \ref{Theo10Gamma} reads
\begin{eqnarray}\label{EqborneChi}
\chi & = & \frac{2 \min\po c_1-2c_2,1\pf}{\sigma^2 \po 100\po 2+ \frac2{\sigma^2 } +  \gamma^2 +\po\|\na^2 V\|_\infty+2\|\na^2 W\|_\infty\pf^2\pf \pf^{20 }},
\end{eqnarray}
which is quite rough. When the potentials are quadratic, we can compare this to the real rate, which is known. Suppose $V(x) = \frac a2|x|^2$ and $W(x)=\frac b2|x|^2$ so that  $Z=(X,Y)$ is a generalized Ornstein-Uhlenbeck process that solves
\begin{eqnarray*}
\dd Z &=& -A Z + \sigma \begin{pmatrix}
0\\ I
\end{pmatrix} \dd B
\end{eqnarray*}
with, denoting by $\pi = \frac1N\begin{pmatrix}
1 & \dots & 1\\
\vdots & \ddots & \vdots\\
1 & \dots & 1
\end{pmatrix}$ the orthogonal projector on $\begin{pmatrix}
1 \\
\vdots \\
1  
\end{pmatrix} $,
\[A \ = \ \begin{pmatrix}
0 & -I\\
a I + b\po I-\pi \pf & \gamma I
\end{pmatrix}.\]
As proved in \cite[Corollary 12]{MonmarcheGamma}, the rate of convergence in the entropic sense of $Z$ to its equilibrium is exactly the spectral gap of $A$. Note that $a I + b\po I-\pi \pf $ is diagonalizable in an orthonormal basis $(u_k)_{k\in\co 1, dN\cf}$, with the corresponding eigenvalues $\lambda_k$ being  either $a+b$ or $a$. Then a vector of the form $(u_k,r u_k)$ for some $r\in\mathbb C$ is an eigenvector of $A$ if only if
\[r \ = \ -\frac{\gamma}{2} \pm \sqrt{\po\frac{\gamma}{2}\pf^2 - \lambda_k},\]
in which case the corresponding eigenvalue is $-r$. It means that, in this quadratic case, Theorem \ref{TheoNLine} holds in fact with
\[
\begin{array}{ccll}
\chi &=& \frac\gamma2 & if \ \po \frac\gamma2\pf^2 <  \min\po a,a+b\pf\\ 
&=&  \frac\gamma2 -\sqrt{\po \frac\gamma2\pf^2 - \min\po a,a+b\pf} & if \ \po \frac\gamma2\pf^2 >  \min\po a,a+b\pf
\end{array}
\]
(in the case $ \gamma^2=  4\min\po a,a+b\pf$,    an additional polynomial factor should be added, see \cite{MonmarcheGamma}). On the other hand, the bound \eqref{EqborneChi} here only reads
\begin{eqnarray*}
\chi & \geq & \frac{2 \min\po a+b,1\pf}{\sigma^2 \po 100\po 2+ \frac2{\sigma^2 } +  \gamma^2 +\po a + 2|b|\pf^2\pf \pf^{20 }}.
\end{eqnarray*}
For instance, if $\sigma=a=b=\gamma=1$, we should have $\chi = \frac12$,
 while from  \eqref{EqborneChi} we only get $\chi \geq 2 \times 10^{-63}$ (again, the computations that lead to \eqref{EqborneChi} were rough but, should we carefully refine them step by step, we would  still miss the target by many orders of magnitude).

\section{Confidence intervals}\label{SectionIntervalle}

From the work of Malrieu on the McKean-Vlasov equation, we obtain the two following Lemmas:

\begin{lem}\label{LemChaosEquilibre}
Under Assumption \ref{Hypo1}, the Vlasov-Fokker-Planck equation \eqref{EqVlasovFP} admits a unique equilibrium $m_{\infty}$ (with normalized mass) which satisfies a log-Sobolev inequality with constant $\eta$ (given by Lemma \ref{LemSobolevmN}), and there exists $K>0$ that depends only on $U,\gamma,\sigma$ and $m_0$ such that for all $N$
\begin{eqnarray*}
\W_2^2\po m_\infty^{(N)},m_\infty^{\otimes N}\pf &\leq& K\\
\| m^{(1,N)}_\infty - m_\infty \|_1 & \leq & \frac{K}{\sqrt N}.
\end{eqnarray*}

\end{lem}
\begin{proof}
According to \cite{Tugaut2015}, a measure $\mu$ is an equilibrium of \eqref{EqVlasovFP} if and only if $\mu(\dd x,\dd y) = \nu(x)\dd x \otimes G(\dd y)$ where $G$ is the Gaussian distribution with variance $\frac{\sigma^2}{2\gamma}$ and $\nu$ is an equilibrium of the McKean-Vlasov equation
\begin{eqnarray*}
\partial_t \rho_t &=& \na \cdot\po \na \rho + \rho_t\int \na U(x,u) \rho(u)\dd u\pf.
\end{eqnarray*} 
Under the convexity condition of Assumption \ref{Hypo1}, according to \cite{Malrieu2001} such an equilibrium $\nu$ is unique and satisfies the same log-Sobolev inequality as does the invariant measure of the corresponding particle system uniformly in $N\geq 1$. Thus by tensorization $m_\infty$ satisfies a log-Sobolev inequality.

Since the second $dN$-dimensional marginals of $m_\infty^{(N)}$ and $m_\infty^{\otimes N}$ are equal, equal to $G^{\otimes N}$, the $\mathcal W_2$ and total variation bounds only concern the first marginals. The $\mathcal W_2$ (resp. total variation) bound   is obtained by letting $t$ go to infinity in \cite[Theorem 1.2]{Malrieu2001} (resp. \cite[Proposition 3.13]{Malrieu2001}) when the initial law is $m_\infty$, which states that $\W_2^2\po m_t^{(N)},m_\infty^{\otimes N}\pf$ is bounded uniformly in $N$ and $t$ (and similarly for the total variation distance). The convergence of $m_t^{(N)}$ to its equilibrium, from Theorem \ref{TheoNLine} together with the Talagrand and Pinsker's Inequalities, concludes.
 
\end{proof}

\begin{lem}\label{LemPropMalrieu}
Under Assumption \ref{Hypo1}, there exists $K$ depending only on $U,\gamma,\sigma$ and $m_0$ such that 
\begin{eqnarray*}
\ent{m_0^{\otimes N}}{m_\infty^{(N)}} & \leq & KN.
\end{eqnarray*}
\end{lem}
\begin{proof}
Without loss of generality, up to some translations, we suppose the potentials $V$ and $W$ are positive and vanish at the origin. Writing $\Psi_N(x,y) = \frac{2\gamma}{\sigma^2}\po U_N(x) + \frac{|y|^2}{2}\pf$, note that from Lemma \ref{LemHessienneUN}, $\Psi_N\ \leq\ K \po|x|^2+|y|^2\pf$ for some $K$ that do not depend on $N$. Then,
\begin{eqnarray*}
& & \ent{m_0^{\otimes N}}{m_\infty^{(N)}} \\
 & = & \int m_0^{\otimes N}   \ln\po m_0^{\otimes N} \pf +  \int m_0^{\otimes N}  \Psi_N   + \ln \int e^{-\Psi_N}\\
&\leq & N  \po\int m_0 \ln m_0  + K \int m_0 \po|x_1|^2+|y_1|^2\pf +\ln \int e^{-\frac{2\gamma}{\sigma^2}\po V(x_1)+\frac{|y_1|^2}{2}\pf}\pf.
\end{eqnarray*}
\end{proof}

\begin{lem}\label{LemCoupleEmpirique}
Let $\nu_1$ and $\nu_2$ be probability laws on $\R^{dN}=\po\R^d\pf^{N}$ which are fixed by any permutation of the  $d$-dimensional coordinates (in other words, if  $(A_{i})_{i\in\cco 1,N\ccf}$ is of law $\nu$, the $A_i$'s are  interchangeable). Let $(A,B)=(A_i,B_i)_{i\in\cco 1,N\ccf}$ be a coupling of $\nu_1$ and $\nu_2$ such that
\[\mathbb E\po |A-B|^2\pf = \mathcal W^2_2(\nu_1,\nu_2).\]
Then
\begin{eqnarray*}
\mathbb E\po \mathcal W^2_2 \po \frac1N\sum_{i=1}^N \delta_{A_i}, \frac1N\sum_{i=1}^N \delta_{B_i}\pf \pf & \leq & \frac{1}{N} \mathcal W^2_2(\nu_1,\nu_2).
\end{eqnarray*} 
\end{lem}
\begin{proof}
Let $I$ be uniformly distributed on $\cco 1,N\ccf$. Then $(A_I,B_I)$ is a coupling of $ \frac1N\sum \delta_{A_i}$ and $ \frac1N\sum \delta_{B_i}$, hence
\begin{eqnarray*}
\mathbb E\po \mathcal W^2_2 \po \frac1N\sum_{i=1}^N \delta_{A_i}, \frac1N\sum_{i=1}^N \delta_{B_i}\pf \pf & \leq & \mathbb{E}\po |A_I-B_I|^2\pf\\
& = & \frac1N \mathbb E\po |A-B|^2\pf.
\end{eqnarray*} 
\end{proof}

\textbf{Remark:}  consider the $\mathcal W_2$ metric on $\mathcal P\po \mathcal P(\R^d)\pf$ when  $\mathcal P(\R^d)$ is itself endowed with the Euclidean $\mathcal W_2$ metric. Let $\Pi_N$ be the application from $  \mathcal P\po \R^{dN} \pf$ to $\mathcal P\po \mathcal P(\R^d)\pf$ defined by $\Pi_N(\nu) = \mathcal L\po\frac1N \sum \delta_{X_i},X\sim \nu\pf$. Lemma \ref{LemCoupleEmpirique} implies that for laws of interchangeable particles,
\[\mathcal W_2^2 \po \Pi_N(\nu_1),\Pi_N(\nu_2)\pf \leq \frac1N\mathcal W^2(\nu_1,\nu_2).\]

\begin{proof}[Proof of Corollary \ref{CorPW2}]
Consider $(Z,\overline{Z})=(Z_i,\overline{Z}_i)_{i\in\cco 1,N\ccf}$ an optimal coupling of $m_t^{(N)}$ and $m_\infty^{\otimes N}$, in the sense
\[\mathbb{E}\po |Z-\overline Z|^2\pf = \mathcal{W}_2^2\po m_t^{(N)},m_\infty^{\otimes N}\pf.\]
Let $M_t^N = \frac1N\sum \delta_{Z_i}$ and $\overline M^N = \frac1N\sum \delta_{\overline Z_i}$, so that
\begin{eqnarray*}
\mathbb P\po \mathcal W_2\po M_t^N , m_\infty \pf\geq \varepsilon \pf & \leq & \mathbb P\po \mathcal W_2\po M_t^N , \overline M^N\pf\geq \frac\varepsilon2 \pf \\
& & +\ \mathbb P\po\mathcal W_2\po \overline M^N,  m_\infty \pf\geq \frac\varepsilon2 \pf.
\end{eqnarray*}
Since the $\overline Z_i$'s are independent, the second term falls within the scope of \cite[Theorem 1.1]{BolleyGuillinVillani} (the log-Sobolev inequality satisfied by $m_\infty$ implies a $T_2$ Talagrand one), which gives an exponential bound for $N$ large enough. As far as the first term is concerned, from Lemma \ref{LemCoupleEmpirique} and the Markov inequality,
\[\mathbb P\po \mathcal W_2\po M_t^N , \overline M^N\pf\geq \frac\varepsilon2 \pf \leq \frac{8}{N\varepsilon^2}\po \W_2^2(m_t^{(N)},m_\infty^{(N)}) + \W_2^2(m_\infty^{(N)},m_\infty^{\otimes N},)\pf.\]
The last term is bounded by $K$ from Lemma \ref{LemChaosEquilibre}. On the other hand the law $m_\infty^{(N)}$ satisfies a log-Sobolev (hence $T_2$) inequality with a constant that does not depend on $N$, which together with Theorem \ref{TheoNLine} and  Lemma \ref{LemPropMalrieu} yields
\begin{eqnarray}\label{EqW2KN}
\W_2^2\po m_t^{(N)},m_\infty^{(N)}\pf & \leq & Ke^{-\chi t}\ent{m_0^{(N)}}{m_\infty^{(N)}}\notag\\
& \leq & K N e^{-\chi t}.
\end{eqnarray}
Altogether we have obtained
\begin{eqnarray*}
\mathbb P\po \mathcal W_2\po M_t^N , m_\infty \pf\geq \varepsilon \pf & \leq &  \frac{K e^{-\chi t}}{\varepsilon^2}  + \frac{K }{N \varepsilon^2}  + e^{-K^{-1}N \varepsilon^2}.
\end{eqnarray*}
\end{proof}

\section{Propagation of chaos}\label{SectionChaos}

The existence and uniqueness of solutions for Equations \eqref{EqVlasovFP} and \eqref{EqParticulNonLinea} under Assumption \ref{Hypo1} are ensured by \cite{Meleard96}. Let $Z_N=(X,Y)$ and $\overline Z_N = (\overline X, \overline Y)$ respectively solve Equations \eqref{EqSystemparticul} and \eqref{EqParticulNonLinea} with the same initial data, of law $m_0$. We first prove a uniform bound on the second moments of both processes:

\begin{lem}\label{LemMoments}
Under Assumptions \ref{Hypo1}, there exists $K>0$ depending only on $U,\gamma,\sigma$ and $m_0$ such that for all $N\geq1,\ t>0$,
\begin{eqnarray*}
\mathbb E\po |X_1(t)|^2 + |\overline X_1(t)|^2  \pf & \leq & K.
\end{eqnarray*}
\end{lem}

\begin{proof}
 
 Let $(X,Y)$ solves 
\[\left\{\begin{array}{rcl}
\dd X &=& Y \dd t\\
\dd Y&=& -\gamma Y\dd t - U_N(X) \dd t + \sigma \dd B
\end{array} \right. \]
where $B$ is $dN$-dimensional Brownian motion. Talay proved  in \cite{Talay} uniform in time moment bounds  for this diffusion, but we need to check the dependency on $N$.

Recall that, as a symmetric matrix, $\na^2U_N \geq \eta$ where $\eta$ is given by Lemma \ref{LemSobolevmN} and in particular depends neither on $t$ nor $N$. Up to a translation we can suppose the unique minimum of $U_N$ is 0 and attained at $x=0$. Consider   the generator
\[\mathcal L_N \ = \ y\cdot\na_x   - \po\na U_N(x)+\gamma y\pf \cdot \na_y   + \frac{\sigma^2}{2} \Delta_{y}\]
(which is the dual in $L^2(m_\infty^{(N)})$ of $L_N$ given by \eqref{EqDefiLN}) and
\[H(x,y) \ = \ U_N(x) + \frac12|y|^2 + \varepsilon x\cdot y\]
for some $\varepsilon >0$. Then
\begin{eqnarray*}
\mathcal L_N  H(x,y) &=& -(\gamma-\varepsilon) |y|^2 + \frac{\sigma^2}{2} dN - \varepsilon \gamma x\cdot y - \varepsilon x \cdot \na U_N(x)\\
& \leq & -(\gamma-\varepsilon - \gamma^2 \sqrt{\varepsilon}) |y|^2 -\varepsilon\po\eta - \sqrt{\varepsilon}\pf|x|^2 + \frac{\sigma^2}{2} dN.
\end{eqnarray*}
Note that on the other hand $ U_N(x) \leq  \po\|\na^2 V\|_\infty+ 2\|\na^2 W\|_\infty\pf |x|^2$. For $\varepsilon$ small enough (and independent from $t$ and $N$) we thus have
\begin{eqnarray*}
\mathcal L_N  H & \leq & -c_1 H  + c_2 N
\end{eqnarray*}
where $c_1$ and $c_2$ are independent from $t$ and $N$. The Gr\"onwall Lemma yields
\begin{eqnarray*}
\mathbb E\po H\po X(t),Y(t)\pf\pf & \leq & \mathbb E\po H\po X(0),Y(0)\pf\pf + \frac{c_2 N}{c_1} \\
& \leq &   KN 
\end{eqnarray*}
where we used interchangeability together with $U_N(x) \leq K |x|^2$. We conclude the case of $(X,Y)$ with interchangeability again and $U_N(x) \geq \eta |x|^2$.

\bigskip

Now let $(\overline{X},\overline{Y})$ solve \eqref{EqParticulNonLinea}  and write
\begin{eqnarray}\label{EqUNbar}
\overline U_N(x) & = &\sum_{i=1}^N \int   U(x_i,v) m_t(v,w) \dd v \dd w,
\end{eqnarray}
 so that
\begin{eqnarray*}
\lefteqn{\partial_t   \mathbb E\po H\po \overline{X}(t),\overline{Y}(t)\pf \pf=}\\
& & \mathbb  E\po \po \mathcal L_N   + \na (U_N -\overline U_N)\na_y\pf H\po \overline{X}(t),\overline{Y}(t)\pf \pf  \\
&\leq & - \frac12 c_1  \mathbb E\po H\po \overline{X}(t),\overline{Y}(t)\pf \pf + c_2 N + c_3  \mathbb  E\po| \na (U_N -\overline U_N)|^2\po \overline{X}(t),\overline{Y}(t)\pf \pf
\end{eqnarray*}
for some $c_3$ independent from $t$ and $N$. The $(\overline X_i,\overline Y_i)$'s being independent with law $m_t$,
\begin{eqnarray}\label{EqNonDiagonal}
& &  \mathbb  E\po| \na (U_N -\overline U)|^2\po \overline{X}(t),\overline{Y}(t)\pf \pf  \notag\\
& = &\frac1{N}\mathbb E\po \left|\sum_{j=1}^N \na W(\overline X_1-\overline X_j)-\int \na W(\overline X_1-u)m_t(u,v)\right|^2 \pf \notag\\
 &=&  \frac1{N}\mathbb E\po \sum_{j=1}^N\left| \na W(\overline X_1-\overline X_j)-\int \na W(\overline X_1-u)m_t(u,v)\right|^2 \pf \notag\\
 &\leq & \frac{\|\na^2 W\|_\infty^2}{N} \mathbb E\po |\overline{X}(t)|^2\pf .
\end{eqnarray}
As a consequence for $N$ large enough
\begin{eqnarray*}
\partial_t   \mathbb E\po H\po \overline{X}(t),\overline{Y}(t)\pf \pf &\leq & -\frac{c_1}{4}  \mathbb E\po H\po \overline{X}(t),\overline{Y}(t)\pf \pf + c_2 N
\end{eqnarray*}
and the conclusion is similar to the case of $(X,Y)$.
\end{proof}

In a first instance, from a classical strategy (the parallel coupling of $Z_N$ and $\overline Z_N$), we prove a propagation of chaos estimate  which badly behaves in time:
\begin{prop}\label{PropChaosGrossier}
Under Assumptions \ref{Hypo1}, there exist $b>0$ (depending only on $U,\gamma$ and $\sigma$)  and $K>0$ (depending only on $U,\gamma,\sigma$ and $m_0$) such that for all $N\geq 1, \ t>0$, if $Z_N$ and $\overline Z_N$ respectively solve Equations \eqref{EqSystemparticul} and \eqref{EqParticulNonLinea} driven by the same Brownian motion, then
\begin{eqnarray*}
\mathbb E\po \left|Z_N(t)-\overline Z_N(t)\right|^2\pf & \leq &  K e^{bt} .
\end{eqnarray*}
\end{prop}
Note that by interchangeability, this reads
\begin{eqnarray*}
\mathbb E\po \left|X_1(t)- \overline X_1(t)\right|^2 + \left|Y_1(t)-\overline   Y_1(t)\right|^2 \pf & \leq &  \frac{K e^{bt}}N 
\end{eqnarray*}
and, recalling that $m_t^{(1,N)}$ stands for the law of $(X_1,Y_1)$, it implies
\begin{eqnarray*}
\mathcal W_2^2\po m_t^{(1,N)},m_t \pf & \leq &  \frac{K e^{bt}}N .
\end{eqnarray*}

\begin{proof}
Let $(x,y) = \po X - \overline X,Y-\overline Y\pf$. The potential $U$ being Lipschitz, it is clear there exists $b'>0$ such that
\begin{eqnarray*}
\partial_t\po |x|^2 + |y|^2\pf & \leq & b' \po |x|^2 + |y|^2\pf \\
& &- \frac2N \sum_{i,j=1}^N y_i\po \na W(X_i-X_j) - \int W(\overline X_i - u) m_t(u,v)\pf  .
\end{eqnarray*}
Decomposing the last factor in 
\begin{eqnarray*}
  \lefteqn{\po \na W(X_i-X_j) -\na W(\overline X_i- \overline X_j)\pf}\\
& &   + \po \na W(\overline X_i- \overline X_j)  -  \int W(\overline X_i - u) m_t(u,v) \pf,
\end{eqnarray*}
and  using $W$ is Lipschitz, the bound \eqref{EqNonDiagonal} and the moment estimate of Lemma~\ref{LemMoments}, we obtain
\begin{eqnarray*}
\partial_t \mathbb E\po |x|^2 + |y|^2\pf & \leq & b  \mathbb E \po |x|^2 + |y|^2\pf + K
\end{eqnarray*}
 for some $b>0$, which concludes.
\end{proof}

From this first rough estimate, together with the convergence in large time of the particle system given by Theorem \ref{TheoNLine}, we obtain a first result in large time for the non-linear process:
\begin{prop}\label{PropWassersteinNonlineaire}
Under Assumptions \ref{Hypo1}, there exists $K>0$ depending only on $U,\gamma,\sigma$ and $m_0$ such that for all $  t>0$,
\begin{eqnarray*}
\mathcal W_2^2\po m_t,m_\infty \pf & \leq & K e^{-\chi t}.
\end{eqnarray*}
where $\chi$ is given by Theorem \ref{TheoNLine}.
\end{prop} 
\begin{proof}
From the bound~\eqref{EqW2KN}, together with Proposition \ref{PropChaosGrossier}, Lemma \ref{LemChaosEquilibre} and interchangeability, it comes
\begin{eqnarray*}
& & \mathcal{W}_2^2\po m_t ,m_\infty \pf \\
 & = & \frac1N 
\mathcal{W}_2^2\po m_t^{\otimes N},m_\infty^{\otimes N}\pf \\
& \leq & \frac1N \po
\mathcal{W}_2\po m_t^{\otimes N},m_t^{(N)}\pf + \mathcal{W}_2\po m_t^{(N)},m_\infty^{(N)}\pf  + \mathcal{W}_2\po m_\infty^{(N)},m_\infty^{\otimes N}\pf  \pf^2 \\
& \leq & \frac KN \po
 e^{b t} +  N e^{-\chi t}  \pf 
\end{eqnarray*}
and we can let $N$ go to infinity.
\end{proof}

Combining our previous propagation of chaos results (at equilibrium, or with an exponential prefactor) together with the convergence in large time ones (for the particle system and for the non-linear process) we can now prove the claimed uniform in time propagation of chaos result:

\begin{proof}[Proof of Theorem \ref{ThmChaosUniforme}]
According to Proposition \ref{PropChaosGrossier}, for $t\leq \varepsilon \ln N $ for some $\varepsilon>0$,
\begin{eqnarray*}
\mathcal{W}_2^2\po m_t^{(1,N)},m_t \pf & \leq & \frac{K}{N^{1-b\varepsilon}}.
\end{eqnarray*}
For $t\geq \varepsilon \ln N$, according to Proposition \ref{PropWassersteinNonlineaire}, Theorem \ref{TheoNLine} and Lemma \ref{LemChaosEquilibre},
\begin{eqnarray*}
\mathcal{W}_2^2\po m_t^{(1,N)},m_t \pf
 & \leq & \frac1N \Big( \mathcal{W}_2 \po m_t^{(N)},m_\infty^{(N)} \pf + \mathcal{W}_2 \po m_\infty^{(N)} ,m_\infty^{\otimes N} \pf\\
 & & +\ \mathcal{W}_2 \po m_\infty^{\otimes N},m_t^{\otimes N} \pf\Big)^2\\
& \leq & K \po \frac{1}{N^{\varepsilon\chi }} + \frac 1 N\pf.
\end{eqnarray*}
Taking $\varepsilon = (\chi+b)^{-1}$, we obtain the result with $\alpha = \po 1 + b/\chi\pf^{-1}$.
\end{proof}

One could expect the result to hold with $\alpha=\frac12$, which is the case in the space-homogeneous Mc Kean-Vlasov equation. Nevertheless,  since the $b$ obtained in the proof of Proposition \ref{PropChaosGrossier} is clearly greater than 1, the best lower bound we could get on $\alpha$ would be less than $\chi$ given in \eqref{EqborneChi}, which is pretty bad. Note that in the quadratic case, the results of \cite{BolleyGentilGuillin} applies, so that Theorem \ref{ThmChaosUniforme} holds with $\alpha = \frac12$. 

\section{The Vlasov-Fokker-Planck equation}\label{SectionVFP}

\begin{lem}\label{LemEntropieBorne}
Under Assumptions \ref{Hypo1} there exist $K$ depending only on $U,\gamma,\sigma$ and $m_0$ such that for all $N\geq 1,\ t>0$,
\begin{eqnarray*}
\| m_t^{(1,N)} - m_t \|_1 & \leq & \frac{ K\sqrt t}{ N^{\frac\alpha4}}  
\end{eqnarray*}
where $\alpha$ is given by Theorem \ref{ThmChaosUniforme}.
\end{lem}

\begin{proof}
We follow the idea of \cite[Lemma 3.15]{Malrieu2001}, namely we compute the derivative of
\begin{eqnarray*}
F(t) & = &   \ent{m_t^{(N)}}{m_t^{\otimes N}}  .
\end{eqnarray*}
To do so, let $u_1 = m_t^{(N)}$, $u_2 = m_t^{\otimes N}$,
\[b_1(x,y) = \begin{pmatrix}
y \\ -\gamma y -  \na_x U_N(x)  
\end{pmatrix},\hspace{25pt}
b_2(x,y) = \begin{pmatrix}
y \\ -\gamma y -\na_x \overline U_{N}(x)
\end{pmatrix}
\]
(where $\overline U_N$ is given by \eqref{EqUNbar}) and $L_i f = - \na \cdot \po b_i  f \pf + \frac{\sigma^2}{2} \Delta_y f$ for $i=1,2$. With these notations, $\partial_t \po u_i\pf = L_i u_i$, and the dual in the Lebesgue sense of $L_i$ is $L_i' = b_i\cdot \na + \frac{\sigma^2}{2}  \Delta_y$. From the conservation of the mass of $u_1$, we get
\[  0 = \partial_t\po \int \frac{u_1}{u_2}u_2 \pf = \int \po L_1 u_1  - \frac{u_1}{u_2}L_2 u_2 +  L_2'\po\frac{u_1}{u_2}\pf u_2 \pf .  \]
Since $L_1'$ is a diffusion operator with carr\'e  du champ operator $\Gamma f = \frac{\sigma^2}{2}  |\na_y f|^2$ (see \cite[p.20 \& 42]{BakryGentilLedoux} for the definitions),
\[ u_1 L_1' \ln\po \frac{u_1}{u_2} \pf = u_1 \frac{L_1'\po\frac{u_1}{u_2}\pf}{\frac{u_1}{u_2}} - u_1\frac{\Gamma\po \frac{u_1}{u_2}\pf}{\po \frac{u_1}{u_2}\pf^2} = u_2 L_1'\po\frac{u_1}{u_2}\pf - u_1 \Gamma\po \ln \frac{u_1}{u_2} \pf.\]
Using both these relations,
\begin{eqnarray*}
\partial_t \po \int \ln\po\frac{u_1}{u_2}\pf u_1\pf & = & \int \po \frac{L_1 u_1}{u_1} - \frac{L_2 u_2}{u_2} + L_1' \ln\po\frac{  u_1}{u_2}\pf\pf u_1\\
& = & \int -\Gamma\po \ln \frac{u_1}{u_2} \pf u_1 + u_2 L_1'\po\frac{u_1}{u_2}\pf - u_2 L_2'\po\frac{u_1}{u_2}\pf \\
& = &   \int -\Gamma\po \ln \frac{u_1}{u_2} \pf u_1 +  (b_1-b_2)\cdot \na \ln \po\frac{u_1}{u_2}\pf u_1.
\end{eqnarray*}
Applying Young's Inequality, we get
\begin{eqnarray*}
F'(t) & \leq & \frac1{2\sigma^2} \int \left| \na  U_N(x)  - \na  \overline U_N(x) \right|^2 m_t^{(N)}\\
& = & \frac{N}{2\sigma^2} \mathbb E\po \left|\frac1N \sum_{j=1}^N \na  W(X_1-X_j) - \int \na W(X_1-v) m_t(v,w)\right|^2\pf
\end{eqnarray*}
by interchangeability. Developing the square of the sum, the diagonal terms are bounded by
\begin{eqnarray*}
  \| \nabla ^2 W\|_\infty^2 \po \mathbb  E\po  |X_j|^2\pf + \int |v|^2 m_t(v,w) \pf &\leq & K
\end{eqnarray*}
where we used  Lemma \ref{LemMoments}. For the extra-diagonal terms, we consider an optimal coupling $\overline Z_N=(\overline X,\overline Y)$ and $Z_N=(X,Y)$ in the sense the law of $\overline Z_N$ is $m_t^{\otimes N}$ and
\begin{eqnarray*}
\mathbb E\po \left|\overline Z_N(t)-Z_N(t)\right|^2\pf &=& \mathcal W_2^2\po m_t^{\otimes N},m_t^{(N)}\pf
\end{eqnarray*} 
and write
\begin{eqnarray*}
& & \po \na  W(X_1-X_j) - \int \na  W(X_1-v) m_t \pf  \po  \na  W(X_1-X_k) - \int \na  W(X_1,v) m_t \pf\\
 & = & \po \na  W(X_1-X_j) - \na  W(X_1-\overline X_j)\pf\po  \na  W(X_1-X_k) - \int \na  W(X_1-v) m_t \pf\\
 & & + \po \na  W(X_1-\overline X_j) - \int \na  W(X_1-v) m_t \pf\po  \na  W(X_1-X_k) - \na  W(X_1-\overline X_k)\pf\\
 & & + \po \na  W(X_1-\overline X_j) - \int \na  W(X_1-v) m_t \pf \po \na  W(X_1-\overline X_k) - \int \na  W(X_1-v) m_t \pf.
\end{eqnarray*}
The $\overline X_i$'s being independent with law the first marginal of $m_t$, the expectation of the third term vanishes, while the expectations of the two other terms is bounded by the Cauchy-Schwarz inequality, interchangeability and Theorem~\ref{ThmChaosUniforme} by
\begin{eqnarray*}
  \| \nabla ^2 W\|_\infty^2 \sqrt{\mathbb  E\po  |X_1 - \overline{X}_1|^2\pf \po \mathbb  E\po  |X_1|^2\pf + \mathbb  E\po  |\overline{X}_1|^2\pf  \pf} &\leq &  \frac{K}{N^{\frac\alpha2}}.
\end{eqnarray*}
 Hence $F'(t) \leq KN^{1-\frac\alpha2}$ and moreover $F(0) = 0$, so that 
we conclude by the Csisz{\'a}r's inequality which reads
\[\ent{\mu^{(1,N)}}{\nu} \leq \frac{2}{N} \ent{\mu^{(N)}}{\nu^{\otimes N}}\]
when  $\mu^{(N)}$ is an interchangeable probability law (see \cite{Csiszar}) and the Pinsker's one.
\end{proof}

From this last propagation of chaos estimate (with a not-so-bad behaviour in time) together with the convergence in large time of the particle system, we are finally able to recover the convergence of the total variation distance for the Vlasov-Fokker-Planck equation:
\begin{proof}[Proof of Theorem \ref{TheoNonLineaire}]
From the Pinsker's inequality, Theorem~\ref{TheoNLine} and  Lemma~\ref{LemPropMalrieu},
\begin{eqnarray*}
\| m_t^{(1,N)}-m_\infty^{(1,N)} \|_1^2 & = & \underset{f\in L^\infty\po\R^{2d}\pf}\sup \int f \po m_t^{(1,N)}-m_\infty^{(1,N)}\pf \\
& \leq & \underset{f\in L^\infty\po\R^{2dN}\pf}\sup \int f \po m_t^{(N)}-m_\infty^{(N)}\pf \\
& = & \| m_t^{(N)}-m_\infty^{(N)} \|_1^2 \\
& \leq & 2 \ent{m_t^{(N)}}{m_\infty^{(N)}} \\
& \leq & K Ne^{-\chi t}.
\end{eqnarray*}

By Lemmas \ref{LemChaosEquilibre} and \ref{LemEntropieBorne},
\begin{eqnarray*}
\| m_t - m_\infty \|_1 & \leq & \| m_t - m_t^{(1,N)} \|_1 + \| m_t^{(1,N)} - m^{(1,N)}_\infty \|_1 + \| m^{(1,N)}_\infty - m_\infty \|_1 \\
& \leq & \frac{K\sqrt{ t}}{  N^{\frac\alpha 4}} + \sqrt{K N} e^{- \frac12\chi t} + \frac{K  }{\sqrt N}
\end{eqnarray*}
and we take $N$ of order $e^{\frac{2\chi t}{\alpha + 2}}.$
\end{proof}
\textbf{Remarks:} in turn this leads to a self-improvement of Lemma \ref{LemEntropieBorne} by writing
\begin{eqnarray*}
\| m_t^{(1,N)}-m_t \|_1 & \leq & \| m_t^{(1,N)}-m_\infty^{(1,N)} \|_1 +\| m_\infty -m_\infty^{(1,N)} \|_1 +\| m_t -m_\infty  \|_1\\
&\leq & K \po e^{-\frac12\chi' t} + \frac{1}{N}\pf
\end{eqnarray*}
which, used for, say, $t\geq N^{\frac{\alpha}{8}}$ while Lemma \ref{LemEntropieBorne} is used for small times, reads
\begin{eqnarray*}
\| m_t^{(1,N)}-m_t \|_1 & \leq & \frac{K}{N^{\frac\alpha8}}
\end{eqnarray*}
for all $N\geq1,\ t>0$.

\bigskip

In the quadratic case, we obtain a rate of convergence equal to $\frac14\chi$, where $\chi$ is given at the end of Section \ref{SectionSystemedeParticule}.

\bigskip

\textbf{Acknowledgements:} The author would like to thank Arnaud Guillin and Florent Malrieu for fruitful discussions.





\end{document}